\documentclass[amscd,amssymb,verbatim]{amsart}
\usepackage{amssymb,amsfonts,amsmath,amscd}

\theoremstyle{plain}
\newtheorem{Thm}{Theorem}[section]
\newtheorem{Lem}[Thm]{Lemma}
\newtheorem{Prop}[Thm]{Proposition}
\newtheorem{Cor}[Thm]{Corollary}

\theoremstyle{definition}
\newtheorem{Def}[Thm]{Definition}

\theoremstyle{remark}

\newcommand{\B}{{\mathcal L}}
\newcommand{\C}{\mathbb{C}}

\newcommand{\cH}{{\mathcal H}}

\newcommand{\cK}{{\mathcal K}}

\newcommand{\M}{\mathbb{M}}
\newcommand{\N}{\mathbb{N}}

\newcommand{\imp}{\Rightarrow}

\newcommand{\limind}{\lim_{\rightarrow}}

\newcommand{\Hi}{\mathcal{H}}
\renewcommand{\L}{\mathcal{L}}

\newcommand{\pf}{\noindent{\mbox{\textbf{Proof}.\,}} }

\newcommand{\cst}{C*}
\theoremstyle{definition}

\newtheorem{definition}[Thm]{Definition}

\newtheorem{rem}[Thm]{Remark}
\newtheorem{rems}[Thm]{Remarks}

\theoremstyle{remark}

\begin{document}
\title[Irreducible Representations]{Irreducible Representations of Inner Quasidiagonal C*-Algebras}
\author{Bruce Blackadar and Eberhard Kirchberg}
\address{Department of Mathematics/084 \\ University of Nevada, Reno \\ Reno, NV 89557, USA \\ Institut f{\"u}r Mathematik,
Humboldt Universit{\"a}t zu Berlin \\
Unter den Linden 6,
D-10099 Berlin, Germany}
\email{bruceb@unr.edu\,,\ kirchbrg@mathematik.hu-berlin.de}
\thanks{This work was completed while both authors were visitors at the Fields Institute, Toronto, Canada.}

%\keywords{C*-algebra, partial isometry, Cuntz-Krieger algebra}
%\subjclass{Primary: 46L05; Secondary: 19A10}
\date{November 29, 2007}

\maketitle
\begin{abstract}
It is shown that a separable C*-algebra is inner quasidiagonal if and only if it has a separating family of
quasidiagonal irreducible representations.  As a consequence, a separable C*-algebra is a strong NF
algebra if and only if it is nuclear and has a separating family of quasidiagonal irreducible representations.
We also obtain some permanence properties of the class of inner quasidiagonal C*-algebras.
\end{abstract}

\section{Introduction}

This article is the long-delayed third installment in the authors' study of generalized inductive limits of
finite-dimensional C*-algebras.

The basic study of such generalized inductive limits was begun in \cite{BlackadarKGeneralized}, where the classes of MF algebras,
NF algebras, and strong NF algebras were defined, and a number of equivalent characterizations of each class given.
In particular, a (necessarily separable) C*-algebra is a {\em strong NF algebra} if it can be written as a
generalized inductive limit of a sequential inductive system of finite-dimensional C*-algebras in which the connecting
maps are complete order embeddings (and asymptotically multiplicative in the sense of \cite{BlackadarKGeneralized}).  An {\em NF algebra}
is a C*-algebra which can be written as the generalized inductive limit of such a system, where the connecting maps
are only required to be completely positive contractions.  An NF algebra is automatically nuclear (and separable).
It was shown that a separable C*-algebra is an NF algebra if and only if it is nuclear and quasidiagonal.

It was not shown in \cite{BlackadarKGeneralized} that the classes of NF algebras and strong NF algebras are distinct.  Our second
paper \cite{BlackadarKInner} used the notion of inner quasidiagonality to distinguish them.  We made the following definition,
a slight variation of Voiculescu's characterization of quasidiagonal C*-algebras \cite{VoiculescuNote}:

\begin{Def}
A C*-algebra $A$ is {\em inner quasidiagonal} if, for every $x_1,\dots,x_n\in A$ and $\epsilon>0$, there is a
representation $\pi$ of $A$ on a Hilbert space $\cH$, and a finite-rank projection $P\in\pi(A)''$ such that
$\|P\pi(x_j)-\pi(x_j)P\|<\epsilon$ $\|P\pi(x_j)P\|>\|x_j\|-\epsilon$ for $1\leq j\leq n$.
\end{Def}

Voiculescu's characterization of quasidiagonality is the same with the requirement that $P\in\pi(A)''$ deleted.
We then proved that a separable C*-algebra is a strong NF algebra if and only if it is nuclear and inner quasidiagonal.

The principal shortcoming of this result is that it is often difficult to determine directly from the definition whether a C*-algebra
is inner quasidiagonal, although we were able to give examples of separable nuclear C*-algebras which are quasidiagonal but
not inner quasidiagonal, hence of NF algebras which are not strong NF.  It is immediate from the definition that a
C*-algebra with a separating family of quasidiagonal irreducible representations is inner quasidiagonal, and in \cite{BlackadarKInner}
we established some very special cases of the converse which were sufficient to yield the examples.

The main result of the present article is the full converse in the separable case:

\begin{Thm}\label{MThm}
A separable C*-algebra is inner quasidiagonal if and only if it has a separating family of quasidiagonal irreducible representations.
\end{Thm}

We thus obtain a characterization of strong NF algebras which is usually much easier to check than the characterization of \cite{BlackadarKInner}:

\begin{Cor}
A separable C*-algebra is a strong NF algebra if and only if it is nuclear and has a separating family of quasidiagonal irreducible representations.
\end{Cor}

Although our initial interest in inner quasidiagonality was through its connection with generalized inductive limits, it now
appears that inner quasidiagonality is of some interest in its own right.  In the final section, we list some permanence properties
of the class of inner quasidiagonal C*-algebras which follow easily from our characterizations:
\begin{enumerate}
\item[1.]  An arbitrary inductive limit (with injective connecting maps) of inner quasidiagonal C*-algebras is inner quasidiagonal.
\item[2.]  A (minimal) tensor product of inner quasidiagonal C*-algebras is inner quasidiagonal.
\item[3.]  The algebra of sections of a continuous field of separable inner quasidiagonal C*-algebras is inner quasidiagonal.
\item[4.]  Inner quasidiagonality is an (SI) property in the sense of \cite[II.8.5]{BlackadarOperator}.
\end{enumerate}

As a result of (1) and (4), the study of inner quasidiagonal C*-algebras can be effectively reduced to studying separable inner
quasidiagonal C*-algebras, to which Theorem \ref{MThm} applies.

\section{Outline of the Proof}

To prove Theorem \ref{MThm}, fix a separable inner quasidiagonal C*-algebra $A$ and a nonzero $x_0\in A$ (we may assume $\|x_0\|=1$
to simplify notation).  We will
manufacture a quasidiagonal irreducible representation $\pi$ of $A$ with $\pi(x_0)\neq0$.

We will construct a sequence $(\cH_n)$ of finite-dimensional Hilbert spaces, embeddings $I_n:\cH_n\to\cH_{n+1}$,
an increasing sequence $(X_n)$ of finite self-adjoint subsets of the unit ball of $A$ containing $x_0$, with union $X$ dense in the unit ball of $A$, and completely positive contractions
$V_n:A\to\B(\cH_n)$ mapping the closed unit ball of $A$ onto the closed unit ball of $\B(\cH_n)$, such that, for all $n$:
\begin{enumerate}
\item[(i)]  $\|I_n^*V_{n+1}(x)I_n-V_n(x)\|<2^{-n-2}$ for all $x\in X_{n+1}$.
%\item[(i$'$)]  $\|V_{n+1}(x)I_n-I_nV_n(x)\|<2^{-n-1}$ for all $x\in X_n$.
\item[(ii)]  $\|V_n(xy)-V_n(x)V_n(y)\|<2^{-n-2}$ for all $x,y\in X_n$.
\item[(iii)]  $V_n(X_{n+1})$ is $2^{-n-2}$-dense in the unit ball of $\B(\cH_n)$, i.e.\ for all $z$ in the unit ball of $\B(\cH_n)$
there is an $x\in X_{n+1}$ with $\|V_n(x)-z\|<2^{-n-2}$.
\item[(iv)] $\|V_1(x_0)\|>3/4$.
\end{enumerate}
%(Actually, (i$'$) follows automatically from (i) and (ii) (\ref{}).)

Once this tower constructed, we proceed as follows.  Let $\cH=\limind(\cH_n,I_n)$ be the inductive limit Hilbert space, which may be
thought of as the ``union'' of the $\cH_n$.  Let $J_n$ be the natural inclusion of $\cH_n$ into $\cH$.
If $x\in X_m$, then for $n\geq m$ and $\xi,\eta\in J_n\cH_n$ we have
$$|\langle(J_{n+1}V_{n+1}(x)J_{n+1}^*-J_nV_n(x)J_n^*)\xi,\eta\rangle|$$
$$=|\langle(J_nI_n^*V_{n+1}(x)I_nJ_n^*-J_nV_n(x)J_n^*)\xi,\eta\rangle|<2^{-n-2}$$
by (i). So the sequence $(J_nV_n(x)J_n^*)$ converges weakly in $\B(\cH)$ to an operator we call $\pi(x)$.
For $\xi\in J_m\cH_m$, for $n\geq m$ we have
$$\|J_{n+1}V_{n+1}(x)J_{n+1}^*\xi\|\geq\|J_nJ_n^*J_{n+1}V_{n+1}(x)J_{n+1}^*\xi\|$$
$$=\|J_{n+1}I_n^*V_{n+1}(x)I_nJ_n^*\xi\|\geq\|J_nV_n(x)J_n^*\xi\|-2^{n-2}$$
and thus $\|\pi(x)\xi\|\geq\limsup\|J_nV_n(x)J_n^*\xi\|$.  So $J_nV_n(x)J_n^*\to\pi(x)$ strongly (cf.\ \cite[I.1.3.3]{BlackadarOperator}).
If $x,y\in X$, it follows from (ii) and joint strong continuity of multiplication on bounded sets that $\pi(xy)=\pi(x)\pi(y)$.
Since $X$ is dense in the unit ball of $A$ and each $V_n$ is a contraction, $(J_nV_n(x)J_n^*)$ converges strongly for each $x\in A$ to an operator
on $\cH$ we call $\pi(x)$, and $\pi$ is linear, completely positive, contractive, and multiplicative, hence a *-representation of $A$ on $\cH$.
For each $m\in\N$ and $x\in X_m$, we have $\|\pi(x)J_n-J_nV_n(x)\|<2^{-n}$ for all $n\geq m$.

To show that $\pi$ is irreducible, suppose $\xi,\eta,\zeta\in\cH$ are unit vectors and $\epsilon>0$.  Choose $m$ with $2^{-m}<\epsilon/4$,
and for some $n\geq m$ choose unit vectors $\tilde\xi ,\tilde\eta, \tilde\zeta \in\cH_n$ with $\|\xi-J_n\tilde\xi \|, \|\eta-J_n\tilde\eta \|, \|\zeta-J_n\tilde\zeta\|<\epsilon/4$.
There is a unitary $u\in\B(\cH_n)$ with $u\tilde\xi =\tilde\eta$.  By (iii), there is an
$x\in X_{n+1}$ with $\|V_n(x)-u\|<2^{-n-2}$.  Since $\|I_n^*V_{n+1}(x)I_n-V_n(x)\|<2^{-n-2}$, we have $\|I_n^*V_{n+1}(x)I_n-u\|<2^{-n-1}$.
By iteration, $\|J_n^*\pi(x)J_n-u\|<2^{-n}$ and hence $\|J_n^*\pi(x)J_n\tilde\xi-\tilde\eta\|<2^{-n}$.  Then
$$|\langle\pi(x)\xi-\eta,\zeta\rangle|\leq\|\zeta-J_n\tilde\zeta\|+|\langle \pi(x)\xi-\eta,J_n\tilde\zeta \rangle|=\|\zeta-J_n\tilde\zeta\|+|\langle J_n^*(\pi(x)\xi-\eta),\tilde\zeta \rangle|$$
$$\leq \|\zeta-J_n\tilde\zeta\|+\|\xi-J_n\tilde\xi\|+\|\eta-J_n\tilde\eta\|+|\langle J_n^*\pi(x)J_n\tilde\xi-\tilde\eta,\tilde\zeta\rangle|<\epsilon$$
%Because $\|x\|\leq1$ and $u$ is unitary, it follows that $\|(1-I_nI_n^*)V_{n+1}(x)I_n\|<2^{-n/2}$ and hence $\|V_{n+1}(x)I_n-I_nV_n(x)\|<2^{-n/2+1}$.
%So $\|\pi(x)J_n-J_nV_n(x)\|<2^{-n/2+2}$.  Then
%$$\|\pi(x)\xi-\eta\|\leq\|\pi(x)(\xi-J_n\tilde\xi )\|+\|(\pi(x)J_n-J_nV_n(x))\tilde\xi \|+\|J_n(V_n(x)-u)\tilde\xi \|$$
%$$+\|J_n(u\tilde\xi -\tilde\eta )\|+\|J_n\tilde\eta -\eta\|$$
%$$\leq\|\xi-J_n\tilde\xi \|+\|\pi(x)J_n-J_nV_n(x)\|+\|V_n(x)-u\|+\|J_n\tilde\eta -\eta\|$$
%$$<\frac{\epsilon}{4}+2^{-n+1}+2^{-n/2+2}+\frac{\epsilon}{4}<\epsilon$$
and so (fixing $\xi$ and $\eta$ and letting $\zeta$ vary) $\eta$ is in the weak closure of $\pi(A)\xi$.

To show that $\pi$ is quasidiagonal, let $P_n=J_nJ_n^*$ be the projection of $\cH$ onto $J_n\cH_n$.  Then
$P_n$ has finite rank, $P_n\to1$ strongly, and, for $x\in X_m$ and $n\geq m$,
$$\|P_n\pi(x)-\pi(x)P_n\|=\max(\|(1-P_n)\pi(x)P_n\|,\|P_n\pi(x)(1-P_n)\|)$$
$$=\max(\|(1-P_n)\pi(x)P_n\|,\|(1-P_n)\pi(x^*)P_n\|)$$
$$\|(1-P_n)\pi(x)P_n\|\leq\|\pi(x)P_n-J_nV_n(x)J_n^*\|+\|P_nJ_nV_n(x)J_n^*-P_n\pi(x)P_n\|$$
$$\leq \|\pi(x)J_n-J_nV_n(x)\|+\|J_nV_n(x)-\pi(x)J_n\|<2^{-n+1}$$
since $P_nJ_n=J_n$.   Similarly, $\|(1-P_n)\pi(x^*)P_n\|<2^{-n+1}$ since $x^*\in X_m$.

Finally, note that $\|J_1V_1(x_0)J_1^*\|=\|V_1(x_0)\|>3/4$ and $\|J_1^*\pi(x_0)J_1-V_1(x_0)\|<1/2$, so $\|\pi(x_0)\|>1/4$.

\section{Pure Matricial States}

In this section $A$ will be a general C*-algebra, not the specific C*-algebra of Section 2.

\begin{definition}
\label{def:pure-matr-state}
A \emph{pure matricial $n$-state}
on a \cst--algebra $A$ is a
completely positive contraction
$V\colon A\to \M_n=M_n(\C )$
such that there is an irreducible
representation $\pi\colon A\to \L (\Hi )$
and an isometry $I\colon \C^n\to \Hi$
such that $V(a)=I^*\pi(a)I$ for
$a\in A$. (Note that $\pi$ and $I$ are uniquely
determined up to unitary equivalence (of the $\Hi$) via the Stinespring dilation.)
\end{definition}

\begin{rems}\label{rems:pure-matr-state}
(i) We have chosen a definition which
is easily applicable for our needs.
There are other characterizations,
which are elementary functional
analysis exercises,
e.g.\ $V\colon A\to \M_n$ maps the
\emph{open} unit ball onto the
\emph{open} unit ball of $\M_n$ and
$V$ is an extreme point of the
convex set $CPC(A,\M_n)$
of completely positive
contractions $T\colon A\to \M_n$.
But, even for $A=\M_n$ the
extreme points of the unital maps in
of $CPC(A,\M_n)$  are in general
not pure
$n$-states, i.e.\ are not automorphisms
of $\M_n$ (for example, if $\phi$ is a
pure state of $A$, then the map $x\mapsto\phi(x)1$
from $A$ to $\M_n$ is an extreme point of $CPC(A,\M_n)$,
but not a pure matricial state of $A$).

\smallskip
\noindent
(ii)
If one applies the Kadison Transitivity
Theorem (see e.g.\ \cite[II.6.1.12]{BlackadarOperator}) to
the irreducible representation $\pi$
and the image of $I$, then one can
see that the restriction of $V$
to the multiplicative domain
$$
A^V:=
\{ a\in A\,: \, V(ab)=V(a)V(b)
\, \textrm{and} \,
V(ba)=V(b)V(a)\,,\, b\in A \}
$$
is an epimorphism from $A^V$ onto
$\M_n$, cf.~ \cite[3.4]{BlackadarKInner}.
In particular, $V$ maps the
closed unit ball of $A$
onto the closed unit ball
of $\M_n$, and hence maps the open unit ball of $A$
onto the open unit ball of $\M_n$.

\smallskip
\noindent
(iii) Up to unitary equivalence
(up to an automorphism of $\M_n$)
a pure matricial $n$-state is defined
by a projection $p$ in the
socle of the second conjugate $A^{**}$
of $A$ with $pAp\cong \M_n$:
Consider  the support projection
$z$ of the normal extension
$\overline{\pi}\colon A^{**}\to \L(\Hi )$ of
$\pi$ in the center of
$A^{**}$. The restriction $\varphi$ of
$\overline{\pi}$ to $A^{**}z$ defines
an isomorphism from  $A^{**}z$ onto  $\L (\Hi )$.
Let $p_V:=\varphi^{-1}(II^*)$; then
$b\mapsto I^*\varphi (b)I$ is an
isomorphism $\lambda _V$
from $p_VAp_V=p_VA^{**}p_V$
onto $\M_n$, such that
$V(b)=\lambda _V(p_Vbp_V)$ for $b\in A$.

Note that $p_V$ is just the support
projection of
the normal extension $\overline{V}$
of $V$ to $A^{**}$. Projections corresponding
to disjoint pure matricial states have orthogonal
central supports.

\smallskip
\noindent
(iv) A pure matricial $n$-state  $V\colon A\to \M_n$ always extends
to a unital pure matricial
$n$-state on the unitization $\tilde{A}$
of $A$  (see (ii)). In particular, $V$ is
unital if $A$ is unital.

\smallskip
\noindent
(v) Pure matricial $n$-states are in
obvious 1-1-correspondence with
those pure states $\eta$ on $A\otimes \M_n$
which have the additional
property that its (unique) extension
$\tilde{\eta}$ to $\tilde{A}\otimes \M_n$
satisfies
$\tilde{\eta}(1\otimes b)=\tau(b)$,
where $\tau$ denotes the tracial state
on $\M_n$. The correspondence is
given by $\eta _V(a\otimes b):=(1/n) \mathrm{Tr}(V(a)b^\top )$
where $b^\top $
denotes the \emph{transposed} matrix of $b$.
Not every pure state of $A\otimes\M_n$ has this property
(cf.\  (i)).

\noindent
[To see that such an $\eta_V$ extends to a pure state on $\tilde{A}\otimes\M_n$, note that
$a\otimes  b\mapsto (1/n) \mathrm{Tr}(b^\top a)$
defines a pure state $\eta$ on $\M_n\otimes \M_n$.
Conversely, every pure state $\eta$  on  $A\otimes \M_n$
is given by an irreducible representation
$\pi\colon A\to \L (\Hi)$ and a map $I\colon \C^n\to \cH$ such that
$\eta(a\otimes b)=(1/n)\mathrm{Tr}(V(a)b^\top)$ for
$a\in A$ and $b\in \M_n$, where $V(a):= I^*\pi(a)I$.
The condition $\eta(1\otimes  b)=(1/n)\mathrm{Tr}(b)$
implies that $I$ is an isometry.]

Here is an alternate way of viewing the situation.  A pure state on $A\otimes\M_n$ is a vector state from an
irreducible representation of $A\otimes\M_n$.  Up to unitary equivalence, every irreducible representation
of $A\otimes\M_n$ is of the form $\rho\otimes\sigma$, where $\rho$ is an irreducible representation of $A$
on a Hilbert space $\cH$, and $\sigma$ is the standard representation of $\M_n$ on $\C^n$.  Let $\{\zeta_1,\dots,\zeta_n\}$
be the standard basis for $\C^n$.  Then every unit vector in $\cH\otimes\C^n$ can be written in the form
$\sum_{j=1}^n \alpha_j(\xi_j\otimes\zeta_j)$, where the $\xi_j$ are unit vectors in $\cH$, $\alpha_j\geq0$,
and $\sum \alpha_j^2=1$; the representation is unique if all $\alpha_j>0$.  Then the vector state from this
vector corresponds to a pure matricial $n$-state on $A$ if and only if the $\xi_j$ are mutually orthogonal and all $\alpha_j$ are equal to $n^{-1/2}$.

\smallskip
\noindent
(vi) Every pure matricial $n$-state
$V\colon A\to \M_n$
on $A\subset B$ extends to
a pure matricial $n$-state
$V_e\colon B\to \M_n$:
Simply extend the pure state
$\tilde{\eta}$
on $\tilde{A}\otimes \M_n$ to
a pure state on $\tilde{B}\otimes \M_n$.

If $T\colon A\to B$ is completely isometric
\emph{and} completely positive then
there is a pure matrical state $W\colon B\to \M_n$
with $W\circ T=V$.

\noindent
[Indeed: $T$ extends to a unital
completely isometric map $T_1\colon A_1\to B_1$ of the
\emph{outer} unitizations $A_1$ and $B_1$.  An extremal extension
of the extremal state on $T_1(A_1)\otimes \M_n \subset B_1\otimes \M_n$
-- related to $V$ -- defines the desired extension of $V\circ T^{-1}$ to
all of $B_1$ by (iv).]

\smallskip
\noindent
(vii)  Since up to unitary equivalence the standard representation of $\M_n$ on $\C^n$
is the only irreducible representation of $\M_n$, every pure matricial state $V:\M_n\to\M_m$
(where necessarily $m\leq n$) is the compression of the identity representation to an
$m$-dimensional subspace of $\C^n$, i.e.\ there is a unique isometry $I:\C^m\to\C^n$
with $V(x)=I^*xI$ for all $x\in\M_n$, where $\M_n$ and $\M_m$ are identified with $\B(\C^n)$
and $\B(\C^m)$ in the standard way.
\end{rems}

The next result is the crucial technical tool needed for construction of the tower.

\begin{Lem}\label{lem:A-to-Bmu}
Suppose that
$\{ \psi _\mu \colon A\to B_\mu\,; \,\, \mu \in \Gamma \}$
is a separating family of
\cst--algebra homomorphisms.
Then for every pure matricial $n$-state
$V\colon A\to \M_n$, every $\delta >0$
and every finite subset $F\subset A$
there is a $\nu\in \Gamma$ and a
pure matricial $n$-state
$W\colon B_\nu\to \M_n$
such that
$$ \| W\psi _\nu (x)-V(x)\| <\delta\,
\quad \forall \, x\in F\,.$$
\end{Lem}

\pf
If $n=1$ (the case of pure states) the result is well known (cf.\ \cite[3.4.2(ii)]{DixmierC*Algebras}).  For the general case,
we may assume $A$ and the $B_\mu$ are unital.
Replace $A$ and $B_\mu$ by $A\otimes\M_n$ and $B_\mu\otimes\M_n$, and $\psi_\mu$ by $\psi_\mu\otimes id$.
Let $F\otimes E=\{x\otimes e_{ij}\,:\, x\in F, 1\leq i,j\leq n\}$, where the $e_{ij}$ are the standard matrix units in $\M_n$.
The pure state $\eta_V$ on $A\otimes\M_n$ corresponding to the pure matricial state $V$ on $A$ can
be approximated arbitrarily closely (within $\delta/6n^4$ will do) on $F\otimes E$ by a pure state $\theta$
on $B_\nu\otimes\M_n$ for some $\nu$.  The restriction of $\theta$ to $1\otimes\M_n$ is not (obviously) exactly $\tau$, the tracial state on $\M_n$,
but is at least approximately $\tau$.  We must perturb $\theta$ to make the restriction exactly $\tau$.

When $\theta$ is represented as a vector state with vector $\sum_{j=1}^n\alpha_j(\xi_j\otimes\zeta_j)$ as in \ref{rems:pure-matr-state}(v),
the $\xi_j$ are almost mutually orthogonal and $\alpha_j$ satisfy $|\alpha_j-n^{-1/2}|<\delta/6n^4$ for all $j$.  Let $\varphi$ be the (pure) state of $B_\nu\otimes\M_n$ corresponding
to the vector $\sum_{j=1}^n n^{-1/2}(\tilde \xi_j\otimes\zeta_j)$, where the $\tilde \xi_j$ are obtained from the $\xi_j$ by the Gram-Schmidt process.
We have $\|\tilde \xi_j-\xi_j\|<\delta/3n^3$, so $\|\varphi-\theta\|<\delta/2n^2$ and $\|\varphi(x)-\eta_V(x)\|<\delta/n^2$ for $x\in F\otimes E$.
Then $\varphi=\eta_W$ for some pure matricial $n$-state $W$ on $A$ factoring through $B_\nu$, and $\|W(x)-V(x)\|<\delta$ for all $x\in F$.
\qed

\section{Constructing the Tower}

We now construct the tower used in the proof in Section 2, using the next two lemmas:

\begin{Lem}\label{fl1}
Let $B$ be an inner quasidiagonal C*-algebra, $F$ a finite subset of the unit ball of $B$, $b\in F$, and $\epsilon>0$.  Then there is a pure
matricial state $V:B\to\M_n$ for some $n$, such that $\|V(xy)-V(x)V(y)\|<\epsilon$ for all $x,y\in F$ and $\|V(b)\|>\|b\|-\epsilon$.
\end{Lem}

\begin{proof}
In the separable case, this is just (i) $\imp$ (ii) of \cite[3.7]{BlackadarKInner} (note that there is a misprint in the published statement
of \cite[3.16(ii)]{BlackadarKInner}).  We give the simple argument, which was omitted in \cite{BlackadarKInner}
and which does not require separability.

By the definition of inner quasidiagonality, there is a representation $\pi$ of $B$ on a Hilbert space $\cH$ and a finite-rank
projection $P\in\pi(B)''\subseteq\B(\cH)$ such that $\|P\pi(x)-\pi(x)P\|<\epsilon$ for all $x\in F$ and $\|P\pi(b)P\|>\|b\|-\epsilon$.  The
central support $Q$ of $P$ in $\pi(B)''$ is Type I and is a sum of finitely many minimal central projections $Q_1,\dots,Q_m$.
If $R_1,\dots,R_m$ are minimal projections in $\pi(B)'$ with $Q_j$ the central support of $R_j$, and $P_j=PR_j$, then
$\|P\pi(x)P\|=\max_j\|P_j\pi(x)P_j\|$ for all $x\in B$.  Fix $j$ with $\|P_j\pi(b)P_j\|>\|b\|-\epsilon$.  Then $\rho=\pi|_{R_j\cH}$
is irreducible.  For $x\in B$, let $V(x)=P_j\rho(x)P_j\in\B(P_j\cH)\cong\M_n$, where $n=dim(P_j\cH)$; then $V$ is the desired pure matricial state of $B$:
$\|P_j\pi(x)-\pi(x)P_j\|<\epsilon$ for all $x\in F$, so
$$\|V(xy)-V(x)V(y)\|=\|P_j\pi(xy)P_j-P_j\pi(x)P_j\pi(y)P_j\|$$
$$=\|P_j(P_j\pi(x)-\pi(x)P_j)\pi(y)P_j\|<\epsilon$$
for all $x,y\in F$.
\end{proof}

\begin{Lem}\label{fl2}
Let $B$ be a separable inner quasidiagonal C*-algebra, $F$ a finite subset of $B$, $V:B\to\M_n\cong\B(\C^n)$ a pure matricial state of $B$, and $\epsilon>0$.
Then there is a pure matricial state $W:B\to\M_m\cong\B(\C^m)$ for some $m$ and an isometry $I:\C^n\to\C^m$ such that
\begin{enumerate}
\item[(i)]  $\|W(xy)-W(x)W(y)\|<\epsilon$ for all $x,y\in F$.
\item[(ii)]  $\|I^*W(x)I-V(x)\|<\epsilon$ for all $x\in F$.
\end{enumerate}
\end{Lem}

\begin{proof}
Let $X=\{x_1,x_2,\dots\}$ be a countable dense subset of the unit ball of $B$, and $X_k=\{x_1,\dots,x_k\}$.
For each $k$, apply \ref{fl1} to obtain a pure matricial state $W_k:B\to\M_{m_k}$ such that $\|W_k(xy)-W_k(x)W_k(y)\|<2^{-k}$ for all $x,y\in X_k$,
and $\|W_k(x_k)\|>\|x_k\|-2^{-k}$.  Set $M_k=\M_{m_k}$, $M=\prod_k M_k$, $J=\oplus_k M_k$.   Let $\varphi_k:M\to M_k$ be the $k$'th coordinate map.  Then $\{\varphi_k\}$ is
a separating family of *-homomorphisms on $M$ (since for each $k_0$, $\{x_k\,:\,k\geq k_0\}$ is dense in the unit ball of $B$).

The map $\Psi:b\mapsto(W_1(b),W_2(b),\dots)$ from $B$ to $M$
drops to an injective *-homo\-morphism $\psi$ from $B$ to $M/J$.  By \ref{rems:pure-matr-state}(vi), the pure matricial state $V:B\to\M_n$ extends to a pure matricial state, also called $V$,
from $M/J$ to $\M_n$; $V$ may be regarded as a pure matricial state on $M$ by composing with the quotient map from $M$ to $M/J$.  By \ref{lem:A-to-Bmu}, for some $k$ there is a
pure matricial state $U$ on $M_k$ with $\|U(\varphi_k(x))-V(x)\|<\epsilon$ for all $x\in F$ (where $B$ is identified with $\Psi(B)$).  By \ref{rems:pure-matr-state}(vii), there is an isometry $I:\C^n\to\C^{m_k}$ such that
$U(y)=I^*yI$ for $y\in M_k$.  Set $m=m_k$, $W=W_k$ (note that $W_k(x)=\varphi_k(x)$ for $x\in B\subseteq M$).
\end{proof}

We now construct the tower.  %We will need some overkill in the constants for (i) and (ii) in order to obtain (i$'$).
Let $A$ and $x_0$ be as in Section 2, and let $X$ be a self-adjoint countable dense subset of the unit ball of $A$ containing $x_0$, closed under multiplication.
Enumerate $X$ as
$$X=\{x_0,x_0^*,x_1,x_1^*,\dots,x_n,x_n^*,\dots\}$$
Set $X_1=\{x_0,x_0^*\}$, and by Lemma \ref{fl1} choose a pure matricial $m_1$-state $V_1:A\to\M_{m_1}\cong\B(\C^{m_1})$ of $B=A$ with $F=X_1$, $b=x_0$, and $\epsilon=2^{-3}$,
and set $\cH_1=\C^{m_1}$.

Suppose $X_j,\cH_j=\C^{m_j},V_j, I_j$ have been defined for $1\leq j\leq n$.  Since $V_n$ maps the closed unit ball of $A$ onto the closed unit ball of $\M_{m_n}$,
there is a $k_n$ such that $X_{n+1}:=\{x_0,x_0^*,\dots,x_{k_n},x_{k_n}^*\}$ contains $X_n$ and $X_n^2$ and such that $V_n(X_{n+1})$ is $2^{-n-2}$-dense in the unit ball of $\M_{m_n}$.
By Lemma \ref{fl2} with $B=A$, $F=X_{n+1}$, $V=V_n$, and $\epsilon=2^{-n-2}$ choose a pure matricial $m_{n+1}$-state $V_{n+1}$ of $A$ and an isometry $I_n:\C^{m_n}\to\C^{m_{n+1}}$.

The $X_n,\cH_n,V_n,I_n$ satisfy (i), (ii), (iii), and (iv) by construction.
The tower thus has all the properties required in Section 2, completing the proof of Theorem \ref{MThm}.

\begin{rem}
What about the nonseparable case?  Many parts of the argument have obvious generalizations to the nonseparable case.
The separability hypothesis in Lemma \ref{fl2} should be removable at the cost of some technical complication.
However, it is doubtful that the argument in the proof of Theorem \ref{MThm} can be adapted to the nonseparable case.  For a slight variation of
this argument can be used to give a new proof of the well-known fact that a separable prime C*-algebra is primitive; this is known to be false
in general for nonseparable C*-algebras \cite{WeaverPrime}.
\end{rem}

\section{Inductive Limits}

We do not know whether a general inner quasidiagonal C*-algebra has a separating family of quasidiagonal irreducible representations.  However,
it follows immediately from \cite[3.6]{BlackadarKInner} that every inner quasidiagonal C*-algebra is an inductive limit (with injective connecting maps)
of separable inner quasidiagonal C*-algebras, for which \ref{MThm} holds.  Thus the theory of inner quasidiagonal C*-algebras can be largely reduced
to the separable case.  But to complete this reduction, we must know that an arbitrary inductive limit of inner quasidiagonal C*-algebras is inner quasidiagonal.

In \cite[2.11]{BlackadarKInner}, it was stated that ``it is obvious from the definition'' that an inductive limit of an inductive system of
inner quasidiagonal C*-algebras (with injective connecting maps) is inner quasidiagonal.  In fact, this is not so obvious from the definition,
but it is obvious from the equivalence of (i) and (ii) of \cite[3.7]{BlackadarKInner} in the separable case (see \ref{IndLimThm}).  So to prove that general
inductive limits of inner quasidiagonal C*-algebras are inner quasidiagonal, it suffices to remove the separability hypothesis in this equivalence:

\begin{Prop}\label{Nonsep37}
Let $A$ be a C*-algebra.  Then $A$ is inner quasidiagonal if and only if the following condition is satisfied:

\noindent
For every $a_1,\dots,a_n,b\in A$ and $\epsilon>0$, there is a pure matricial state $V$ of $A$ such that $\|V(a_i)V(a_j)-V(a_ia_j)\|<\epsilon$
for $1\leq i,j\leq n$ and $\|V(b)\|>\|b\|-\epsilon$.
\end{Prop}

The ``only if'' direction follows immediately from Lemma \ref{fl1}.  To prove the converse, we will show by induction on $m$ that if the
statement in \ref{Nonsep37} holds, the following condition $P(m)$ holds for every $m$.  Then, given $a_1,\dots,a_n\in A$ and $\epsilon>0$,
applying $P(n)$ with $b_j=a_j$ shows that $A$ is inner quasidiagonal.

\bigskip

$P(m):$ For every $a_1,\dots,a_n,b_1,\dots,b_m\in A$ and $\epsilon>0$, there are finitely many pure matricial states $V_1,\dots,V_r$ of $A$ such that
\begin{enumerate}
\item[(i)] $\|V_k(a_i)V_k(a_j)-V_k(a_ia_j)\|<\epsilon$ for $1\leq i,j\leq n$, $1\leq k\leq r$.
\item[(ii)]  $\max_k \|V_k(b_j)\|>\|b_j\|-\epsilon$ for $1\leq j\leq m$.
\item[(iii)]  The $V_k$ are pairwise disjoint (the corresponding irreducible representations are pairwise inequivalent).
\end{enumerate}

Note that $P(1)$ is exactly the condition in the statement of \ref{Nonsep37}.

\bigskip

Assume $P(m)$ holds (and thus $P(1)$ also holds), and let $a_1,\dots,a_n,b_1,\dots,b_{m+1}\in A$ and $\epsilon>0$.
Fix $\delta>0$ such that $\delta<\frac{1}{4}$ and
$$2\delta+12\delta(\max_{i,j}\{\|a_i\|,\|b_j\|\})<\epsilon\,.$$
Let $\pi_k$ ($1\leq k\leq r$)
be pairwise inequivalent irreducible representations of $A$ on $\cH_k$ with finite-rank projections $p_k\in\B(\cH_k)$
such that the pure matricial states $U_k(\cdot)=p_k\pi_k(\cdot)p_k$ satisfy
\begin{enumerate}
\item[(i)] $\|U_k(x)U_k(y)-U_k(xy)\|<\delta$ for $x,y\in\{a_1,\dots,a_n,b_1,\dots,b_{m+1}\}$, $1\leq k\leq r$.
\item[(ii)]  $\max_k \|V_k(b_j)\|>\|b_j\|-\delta$ for $1\leq j\leq m$.
\end{enumerate}

If $\pi_k(A)$ contains $\cK(\cH_k)$ (i.e.\ if $\pi_k(A)\cap\cK(\cH_k)\neq\{0\}$), then there is a $c_k\in A_+$, $\|c_k\|=1$,
with $\pi_k(c_k)=p_k$.  If $\pi_k(A)\cap\cK(\cH_k)=\{0\}$, set $c_k=0$.

By $P(1)$ there is an irreducible representation $\pi_{r+1}$ of $A$ on $\cH_{r+1}$ and a finite-rank projection $p_{r+1}\in\B(\cH_{r+1})$
such that the pure matricial state $W(\cdot)=p_{r+1}\pi_{r+1}(\cdot)p_{r+1}$ satisfies
\begin{enumerate}
\item[(i)] $\|W(x)W(y)-W(xy)\|<\delta/2$ for $x,y\in\{a_1,\dots,a_n,b_1,\dots,b_{m+1},c_1,\dots,c_r\}$.
\item[(ii)]  $\|W(b_{m+1})\|>\|b_{m+1}\|-\delta$.
\end{enumerate}

If $\pi_{r+1}$ is not equivalent to any $\pi_k$, $k\leq r$, then we can set $V_k=U_k$ for $1\leq k\leq r$ and $V_{r+1}=W$,
and we are done (since $\delta<\epsilon$).  The difficulty comes when $\pi_{r+1}$ is equivalent to some $\pi_k$, say $\pi_r$
without loss of generality.  In this case, there is an isometry $I$ from $p_{r+1}\cH_{r+1}$ into $\cH_r$ such that
$W(\cdot)=I^*\pi_r(\cdot)I$.

If $\pi_r(A)\cap\cK(\cH_r)=\{0\}$, then (cf.\ \cite{ArvesonNotes}) there is a sequence of isometries $I_t$ from $p_{r+1}\cH_{r+1}$ to
$(1-p_r)\cH_r$ such that
$$W(x)=\lim_{t\to\infty} I_t^*\pi_r(x)I_t\mbox{ for }x\in\{a_1,\dots,a_n,b_1,\dots,b_{m+1}\}\,.$$
For sufficiently large $t$, we may take $V_r(\cdot)=(p_r+I_tI_t^*)\pi_r(\cdot)(p_r+I_tI_t^*)$ and $V_k=U_k$ for $1\leq k\leq r-1$.

The most difficult case is where $\pi_{r+1}$ is equivalent to $\pi_r$ and $\pi_r(A)$ contains $\cK(\cH_r)$.
If $q=II^*$, then we have
$$\|qp_r-p_r q\|=\|q\pi_r(c_r)-\pi(c_r)q\|<\delta\,.$$
By the following lemma, let $\tilde q$ be a projection in $\B(\cH_r)$ with $\tilde q p_k=p_k\tilde q$ and $\|\tilde q -q\|<3\delta$.
Set $\tilde p_r=p_r+\tilde q (1-p_r)$ and $V_r(\cdot)=\tilde p_r\pi(\cdot)\tilde p_r$, and $V_k=U_k$ for $1\leq k\leq r-1$.
These $V_k$ have the desired properties, completing the inductive step and thus the proof of \ref{Nonsep37}.

\begin{Lem}
Let $A$ be a C*-algebra, and $p$ and $q$ projections in $A$.  If $\|qp-pq\|<\epsilon<\frac{1}{4}$, then there is a projection $\tilde q\in A$
with $\|\tilde q -q\|<3\epsilon$ and $\tilde q p=p\tilde q$.  If $r=p+\tilde q (1-p)$, then for every $x\in A$ we have
$$\|rx-xr\|\leq 2\|xp-px\|+2\|xq-qx\|+12\epsilon\|x\|\,.$$
\end{Lem}

\begin{proof}
We may assume $A$ is unital.  We have
$$\|q-[pqp+(1-p)q(1-p)]\|=\|(1-p)qp\|=\|(qp-pq)p\|\leq\|qp-pq\|<\epsilon\,.$$
Also,
$$\|pqp-(pqp)^2\|=\|pqp-pqpqp\|=\|pq(qp-pq)p\|<\epsilon$$
and so $\sigma(pqp)\subseteq[1,\gamma]\cup[1-\gamma,1]$, where $\gamma=\frac{1-\sqrt{1-4\epsilon}}{2}<2\epsilon$ since $\epsilon<\frac{1}{4}$.
Thus by functional calculus there is a projection $r\in pAp$ with $\|r-pqp\|<\gamma<2\epsilon$.  Similarly, there is a projection $s\in (1-p)A(1-p)$
with $\|s-(1-p)q(1-p)\|<2\epsilon$.  If $\tilde q =r+s$, then
$$\|\tilde q -q\|\leq\|\tilde q -[pqp+(1-p)q(1-p)]\|+\|[pqp+(1-p)q(1-p)]-q\|<3\epsilon\,.$$
If $x\in A$, then
$$\|rx-xr\|\leq\|xp-px\|+\|x\tilde q -\tilde q x\|+\|xp\tilde q -p\tilde q x\|$$
$$\leq 2(\|xp-px\|+\|x\tilde q -\tilde q x\|)\leq 2\|xp-px\|+2(2\|x\|\|\tilde q -q\|+\|xq-qx\|)\,.$$
\end{proof}

\begin{Cor}\label{IndLimThm}
An arbitrary inductive limit (with injective connecting maps) of inner quasidiagonal C*-algebras is inner quasidiagonal.
\end{Cor}

\begin{proof}
Let $A=\limind(A_i,\phi_{ij})$, with each $A_i$ inner quasidiagonal.  Regard each $A_i$ as a C*-subalgebra of $A$.  If $a_1,\dots,a_n,b\in A$
and $\epsilon>0$, fix an $A_i$ and elements $\tilde a_1,\dots,\tilde a_n,\tilde b \in A_i$ with $\|a_j-\tilde a_j\|<\delta$ for each $j$
and $\|b-\tilde b\|<\delta$, where $\delta=\epsilon/3\max(1,\|a_1\|,\dots,\|a_n\|,\|b\|)$.
Let $V$ be a pure matricial state of $A_i$ such that $\|V(\tilde a_j)V(\tilde a_k )-V(\tilde a_j\tilde a_k)\|<\delta$
for all $j,k$, and $\|V(\tilde b )\|>\|\tilde b \|-\delta$.  Extend $V$ to a pure matricial state $W$ on $A$.  Then
$\|W(a_j)W(a_k)-W(a_ja_k)\|<\epsilon$ for all $j,k$, and $\|W(b)\|>\|b\|-\epsilon$.  Thus $A$ is inner quasidiagonal by \ref{Nonsep37}.
\end{proof}

\section{Permanence Properties}

We finish by recording some other permanence properties of the class of inner quasidiagonal C*-algebras.  The first one is an easy
consequence of the definition of inner quasidiagonality, and could have been noted in \cite{BlackadarKInner}.

\begin{Prop}\label{TensProd}
The minimal tensor product of inner quasidiagonal C*-algebras is inner quasidiagonal.
\end{Prop}

\begin{proof}
If $A$ and $B$ are inner quasidiagonal and $Z=\{z_1,\dots,z_n\}\subseteq A\otimes_{\min}B$, approximate $z_k$ by an element $\sum_{j=1}^{n_k}x_{jk}\otimes y_{jk}$
of the algebraic tensor product $A\odot B$.
Let $E=\{x_{jk}\}$ and $F=\{y_{jk}\}$.  If $(\pi,P)$ and $(\rho,Q)$ are representations of $A$ and $B$ with projections as in the definition for $E$ and $F$
with sufficiently small $\epsilon$, then $(\pi\otimes\rho,P\otimes Q)$ will be the desired representation for $A\otimes_{\min}B$ and $Z$.
\end{proof}

It is doubtful that the result holds for maximal tensor products.  Note that no separability hypothesis is necessary in \ref{TensProd}.

The next property is an immediate consequence of Theorem \ref{MThm}, and essentially generalizes \cite[3.10]{BlackadarKInner}:

\begin{Prop}
The algebra of sections of a continuous field of separable inner quasidiagonal C*-algebras is inner quasidiagonal.
\end{Prop}

\begin{proof}
Let $\langle A(t)\rangle$ be a continuous field of separable continuous trace C*-algebras over a locally compact Hausdorff space $X$, and
$A$ the C*-algebra of continuous sections vanishing at infinity.  Each fiber $A(t)$ has a separating family of quasidiagonal irreducible
representations by \ref{MThm}, so by composition with the fiber maps from $A$ to the $A(t)$, $A$ also has a separating family of quasidiagonal
irreducible representations, hence is inner quasidiagonal.
\end{proof}

Any C*-subalgebra of a quasidiagonal C*-algebra is quasidiagonal.  This is false for inner quasidiagonality: if $A$ is an NF algebra
which is not strong NF (cf.\ \cite[5.6]{BlackadarKInner}), let $\pi$ be a faithful quasidiagonal representation of $A$ on a Hilbert space $\cH$;
then $\pi(A)+\cK(\cH)$ is inner quasidiagonal (by \ref{MThm} or \cite[5.8]{BlackadarKInner}), in fact strong NF, but the C*-subalgebra $\pi(A)$ is not inner quasidiagonal.  But we have:

\begin{Prop}\label{InnQDSI}
Inner quasidiagonality is an (SI) property in the sense of \cite[II.8.5]{BlackadarOperator}.
\end{Prop}

\begin{proof}
This is just a combination of \ref{IndLimThm} and \cite[3.6]{BlackadarKInner}.
\end{proof}

%\bibliography{cycref}
\bibliographystyle{alpha}

\def\cprime{$'$} \def\cprime{$'$}

\end{document}